\documentclass[runningheads]{llncs}
\usepackage{graphicx}
\usepackage[utf8]{inputenc}
\usepackage[titletoc,title]{appendix}
\usepackage[english]{babel}
\usepackage{amsmath}  
\usepackage{amssymb}   
\usepackage{mathtools}
\usepackage[]{authblk}
\usepackage[linesnumbered, vlined,ruled]{algorithm2e}

\SetCommentSty{mycommfont}
\usepackage[font=bf,skip=\baselineskip]{caption}
\usepackage{etoolbox}
\usepackage{hyperref}
\SetAlCapSkip{1em}
\SetKwInput{KwInput}{Input}
\SetKwInput{KwOutput}{Output}
\SetKwInput{KwOracle}{Oracle}
\SetKwInput{KwNotation}{Notation}
\SetKwInput{KwAssumption}{Assumption}
\usepackage{mathtools}
\SetArgSty{textnormal}
\newtheorem{fact}{Fact}
\newtheorem{obsn}{Observation}
\usepackage[utf8]{inputenc}
\usepackage[dvipsnames]{xcolor}
\usepackage[nottoc]{tocbibind}
\DeclarePairedDelimiter{\floor}{\lfloor}{\rfloor}
\usepackage{enumitem}
\begin{document}
\title{Novel ways of enumerating restrained dominating sets of cycles}
%
%
\author{Sushmita Paul\inst{1} \and Ratanjeet Pratap Chauhan\inst{2} and Srinibas Swain\inst{2}\orcidID{0000-0001-7438-6952}}
\authorrunning{S. Paul et al.}
%
\institute{National Institute of Technology Meghalaya, India,
\email{sushmita.nitm@gmail.com} 
\and
Indian Institute of Information Technology Guwahati, India\\
\email{\{ratanjeet.chauhan,srinibas\}@iiitg.ac.in}
}
\maketitle              
%
 
\begin{abstract}
Let $G = (V, E)$ be a graph. A set $S \subseteq V$ is a restrained dominating set (RDS) if every vertex not in $S$ is adjacent to a vertex in $S$ and to a vertex in $V - S$. The restrained domination number of $G$, denoted by $\gamma_r(G)$, is the smallest cardinality of a restrained dominating set of $G$. Finding the restrained domination number is NP-hard for bipartite and chordal graphs. Let $G_n^i$ be the family of restrained dominating sets of a graph $G$ of order $n$ with cardinality $i$, and let $d_r(G_n, i)=|G_n^i|$. The restrained domination polynomial (RDP) of $G_n$, $D_r(G_n, x)$  is defined as $D_r(G_n, x) = \sum_{i=\gamma_r(G_n)}^{n} d_r(G_n,i)x^i$. In this paper, we focus on the RDP of cycles and have, thus, introduced several novel ways to compute $d_r(C_n, i)$, where $C_n$ is a cycle of order $n$. In the first approach, we use a recursive formula for $d_r(C_n,i)$; while in the other approach, we construct a generating function to compute $d_r(C_n,i)$. 

\end{abstract}

\section{Introduction}
Restrained domination was introduced by Telle et al. \cite{telle1997}. Finding the restrained domination number problem was initially viewed as a vertex partitioning problem. This problem is NP-hard for chordal graphs and bipartite graphs~\cite{gayla1999}, however one of the uses of restrained domination problem is to solve the prisoner and guard problem~\cite{gayla1999}.  Restrained domination problem is well studied and the details of this problem can be seen in~\cite{chen2012}.
 Kayathri et al.~\cite{kayathri2019}  gave a recursive method to find the RDP of cycle of order $n$ with the help of RDPs of paths of order $n-2$ and of order $n$, which is $D_r(C_n,x)=D_r(P_n,x)+3D_r(P_{n-2},x)$, where $P_n$ is a path of order $n$. The
 RDP of $P_n$, in terms of two variables $P_n$ and $x$, is given by  $D_r(P_n,x)=x^2[D_r(P_{n-2},x)+2D_r(P_{n-4},x)+D_r(P_{n-6},x)]$ for $n \geq 6$. However, they did not provide any bounds on $d_r(C_n,k)$. Moreover, in order to compute RDP of $C_n$, their approach forces computation of RDP of several paths of order $\leq n$.  

The main contribution of this paper is a generating function to compute the total number of RDSs of cycle $C_n$ with cardinality $k$, where $\gamma_r(C_n) \leq k \leq n$. We have also introduced a recursive method to compute the total number of restrained dominating sets of $C_n$, with cardinality $i$, using the  number of restrained dominating sets of $C_{n-1}$ and $C_{n-3}$, each with cardinality $i - 1$. We also report a few relations between the coefficients of restrained domination polynomial of cycles in Section~\ref{section4} and conclude by providing some future directions in Section~\ref{section5}. Some of the results presented in this paper are immediate extensions of the results in~\cite{ap2009}, although the problems addressed are different. The results reported in this paper are plain sailing, but some of the proofs used are combinatorially interesting.
Note that we assume $V(C_n)=\{1,2,\cdots,n\}$. 
\section{Restrained Dominating Sets of Cycles}\label{section2}
In this section, we discuss a few structural properties of the restrained dominating sets of cycles.
We first list some existing facts on RDS of cycles. 
The restrained domination number of cycle counted by Domke et al. is given in Fact~\ref{fact1}.
\begin{fact}~\cite{gayla1999}\label{fact1}
$\gamma_r (C_n) = n - 2 \floor{\frac{n}{3}} $. 
\end{fact}

\begin{corollary} \label{coroll1}
    \begin{enumerate}[label=(\roman*)]
        \item For every odd $n \in \mathbb{N}$, $\gamma_r (C_n)$ is an odd number.
        \item For every even $n \in \mathbb{N}$, $\gamma_r (C_n)$ is an even number.
    \end{enumerate}
\end{corollary}

\begin{proof}
\begin{enumerate}[label=(\roman*)]
    \item We know $\gamma_r (C_n) = n - 2 \floor{\frac{n}{3}} $ by Fact~\ref{fact1},  and $n$ is an odd number, therefore, $\gamma_r (C_n)$ is an odd number.

(ii) Can be proved same as (i).
\end{enumerate}
\end{proof}

Kayathri et al. characterized the cardinality of restrained dominating set given in Fact~\ref{fact2}.

\begin{fact}~\cite{kayathri2019} \label{fact2}
Let $S$ be a restrained dominating set of $C_n$. Then $|S| \equiv n\pmod 2$.
\end{fact}

Let $C_n^i$ be the set of restrained dominating sets of $C_n$ with cardinality $i$.

\begin{obsn}\label{lemma1}
$C_j^i = \phi$, if and only if $i > j$ or $i < j-2\floor{\frac{j}{3}}$ or $ (j-i)\%2==1$.
\end{obsn}

We use Observation~\ref{lemma1} to obtain a characterisation on the  sets of restrained dominating sets of cycles.

\begin{proposition}  \label{lemma2}
For all $i \leq n$, the following properties holds 
\begin{enumerate}[label=(\roman*)]
\item If $C_n^i \neq \phi$, then $C_n^{i-1} = C_n^{i+1} = C_{n-1}^i = C_{n+1}^i = \phi$.

\item If $C_{n-1}^{i-1} = C_{n-2}^{i-2} = \phi$, then $C_{n-3}^{i-3} = \phi$.

\item If $C_n^i = \phi$ and $C_{n+2}^{i-2} = \phi$, then $C_{n+1}^{i-1} = \phi$.

\item If $C_n^i \neq \phi$ and $C_{n-2}^i \neq \phi$, then $C_{n-1}^i = \phi$.
\end{enumerate}
\end{proposition}

\begin{proof}
\begin{enumerate}[label=(\roman*)]
    \item  Since $C_n^i \neq \phi$, by Observation~\ref{lemma1}, $ (n-i)\%2=0 $, therefore  $(n-1-i)\%2 \neq 0$ $\Rightarrow$ $C_{n-1}^i = \phi$  and $C_n^{i+1} = \phi$; also $(n+1-i)\%2 \neq 0$ $\Rightarrow$ $C_{n+1}^i \neq \phi$ and $C_n^{i-1} \neq \phi$. \\

    Similarly (ii), (iii) and (iv) can be proved.
\end{enumerate}
\end{proof}

\begin{lemma} \label{lemma3}
If $C_n^i \neq \phi$, then
\begin{enumerate}[label=(\roman*)]

\item $C_{n-1}^{i-1} = C_{n-2}^{i-1} = \phi$ and $C_{n-3}^{i-1} \neq \phi$ if and only if $n=3k$ and $i=k$, for some $k \in \mathbb{N}$.
\item  $C_{n-2}^{i-1} = C_{n-3}^{i-1} = \phi$ and $C_{n-1}^{i-1} \neq \phi$ if and only if $i=n$.

\item $C_{n-1}^{i-1} \neq \phi, C_{n-3}^{i-1} \neq \phi$ and $C_{n-2}^{i-2} = \phi$ if and only if $n=3k+1$ and $i=k+1$.

\item  $C_{n-1}^{i-1} \neq \phi, C_{n-3}^{i-1} \neq \phi$ and $C_{n-5}^{i-1} = \phi$ if and only if $i=n-2$.

\item   $C_{n-1}^{i-1} \neq \phi$ and $C_{n-3}^{i-1} \neq \phi$ if and only if $n-2\floor{\frac{n-1}{3}} \leq i \leq n-2$.

\end{enumerate}
\end{lemma}

\begin{proof}
\begin{enumerate}[label=(\roman*)]
 
    \item $(\Rightarrow) $ Since  $C_{n-1}^{i-1} = C_{n-2}^{i-1} = \phi$, by Observation~\ref{lemma1}, we have, $i-1>n-1$ or, $i-1<n-2-2\floor{\frac{n-2}{3}}$ o, $n-1-i+1$ and $n-2-i+1$ are odd. However, the third case is not possible because two consecutive numbers cannot be odd. If $i-1>n-1$, then $i>n$, then by Observation~\ref{lemma1}, $C_n^i = \phi$, a contradiction. So we have $i-1<n-2-2\floor{\frac{n-2}{3}}$, and $C_n^i \neq \phi$, which is $n-2\floor{\frac{n}{3}} \leq i < n-1-2\floor{\frac{n-2}{3}}$ together, resulting in $n=3k$ and $i=k$, for some $k \in \mathbb{N}$.

  $(\Leftarrow) $ If $n=3k$ and $i=k$ for some $k \in \mathbb{N}$, then by Observation~\ref{lemma1}, we have $C_{n-1}^{i-1} = C_{n-2}^{i-1} = \phi$ and $C_{n-3}^{i-1} \neq \phi$.
  \newline

$(ii), (iii),(iv)$ and $(v)$ can be proved in a similar way as of the proof of $(i)$.
\end{enumerate}
\end{proof}

By using Observation~\ref{lemma1}, Proposition~\ref{lemma2} and Lemma~\ref{lemma3}, we characterize the set of restrained dominating sets of $C_n$.
\begin{theorem} \label{theorem3}
For every $n \geq 4$ and $i \geq n-2\floor{\frac{n}{3}}$,
\begin{enumerate}[label=(\roman*)]

    \item If $C_{n-1}^{i-1} = C_{n-2}^{i-1} = \phi$ and $C_{n-3}^{i-1} \neq \phi$, then $C_n^i = C_n^{\frac{n}{3}} = \big\{ \{1,4,\dots,n-2\},\{2,5,\dots,n-1\},\{3,6,\dots,n\}\big\}$.
    
    \item If $C_{n-2}^{i-1} = C_{n-3}^{i-1} = \phi$ and $C_{n-1}^{i-1} \neq \phi$, then $C_n^i = C_n^n = \{\{1,2,\dots,n\}\}$.
    
    \item If $C_{n-1}^{i-1} \neq \phi, C_{n-3}^{i-1} \neq \phi$ and $C_{n-5}^{i-1} = \phi$, then $C_n^i = C_n^{n-2} = \{ \{1,2,\dots,n\} - \{x,y\}, \forall (x, y) \in E(C_n) \}$.
    
    \item If $C_{n-1}^{i-1} \neq \phi$ and $C_{n-3}^{i-1} \neq \phi$, then for $X_1 \in C_{n-3}^{i-1}$ and $X_2 \in C_{n-1}^{i-1}$\\
    $ C_n^i =  \Big\{  X_1 \cup 
    \begin{cases}
    \{n-2\},  & \quad \text{if } 1 \in X_1 \\
    \{n\} ,& \quad \text{if  1 and 2} \notin X_1 \\
    \{n-1\}, & \quad \text{if  1 and n-3} \notin X_1\\
    \end{cases} \Big\} \cup\\
     \Big\{  X_2 \cup 
    \begin{cases}
    \{n\}, & \quad \text{if } 1 \in X_2 \text{ or 1 and 2} \notin X_2\\
    \{n-1\}, & \quad \text{if  1 and n-1} \notin X_2 \text{, where } X_2 \\
    \end{cases} \Big\}.
    $

\end{enumerate}
\end{theorem}

\begin{proof}
\begin{enumerate}[label=(\roman*)]

    \item Since $C_{n-1}^{i-1} = C_{n-2}^{i-1} = \phi$ and $C_{n-3}^{i-1} \neq \phi$, then by Lemma~\ref{lemma3} $(i)$, we have $n=3k$ and $i=k$, for some $k \in \mathbb{N}$. Therefore $C_n^i = C_n^{\frac{n}{3}} = \big\{ \{1,4,7,\dots,n-2\},\{2,5,8,\dots,n-1\},\{3,6,9,\dots,n\}\big\}$.
    
    \item Since $C_{n-2}^{i-1} = C_{n-3}^{i-1} = \phi$ and $C_{n-1}^{i-1} \neq \phi$, by Lemma~\ref{lemma3}$(ii)$, we have $i=n$. Therefore $C_n^i = C_n^n = \{{\{1,2,\dots,n\}}$\}.
    
    \item Since $C_{n-1}^{i-1} \neq \phi, C_{n-3}^{i-1} \neq \phi$ and $C_{n-5}^{i-1} = \phi$, by  Lemma~\ref{lemma3}$(iv)$, we have $i=n-2$. Therefore $C_n^i = C_n^{n-2}=\{\{1,2,\dots,n\}-\{x,y\}, \forall (x,y) \in E(C_n) \}$.
    
    \item Since $C_{n-1}^{i-1} \neq \phi$ and $C_{n-3}^{i-1} \neq \phi$.
     Suppose that $C_{n-3}^{i-1} \neq \phi$, then there exists an RDS,  $X_1 \in C_{n-3}^{i-1}$. \\Let $Y_1$ = 
 \Big\{  $X_1 \text{ }\cup$
$\begin{cases}
\{n-2\},  & \quad \text{if } 1 \in X_1 \\
    \{n\}, & \quad \text{if  1 and 2} \notin X_1 \\
    \{n-1\}, & \quad \text{if  1 and n-3} \notin X_1 \\
\end{cases}$
\Big\}.\\
For vertices $1, 2, n-3$, by Fact$~\ref{fact2}$, we have vertices $2$ and $1 \notin X_1$, or vertices $1$ and $n-3 \notin X_1$, or vertex $1 \in X_1 $. 
If vertices $2$ and $1 \notin X_1$, then $X_1 \cup \{n\} \in C_n^i $.
If vertices $1$ and $n-3 \notin X_1$, then $X_1 \cup \{n-1\} \in C_n^i $.
If vertex $1 \in X_1$ then $X_1 \cup \{n-2\} \in C_n^i $.
Thus $Y_1 \subseteq C_n^i $.
Now suppose $C_{n-1}^{i-1} \neq \phi$, then there exists an RDS, $X_2 \in C_{n-1}^{i-1} $.\\
Let $Y_2$ =
\Big\{ $X_2 \text{ }\cup$
$\begin{cases}
\{n\}, & \quad \text{if } 1 \in X_2 \text{ or 1 and 2} \notin X_2 \\
    \{n-1\}, & \quad \text{if  1 and n-1} \notin X_2 \\
\end{cases} \Big\}$.\\
If vertices $2$ and $1 \notin X_2$, then $X_2 \cup \{n\} \in C_n^i $.
If vertices $1$ and $n-1 \notin X_2$, then  $X_2 \cup \{1\} \in C_n^i$ but this case is covered when vertex $1 \in X_1$, for $X_1 \in C_{n-3}^{i-1}$, thus, $X_2 \cup \{n-1\} \in C_n^i $. If $1 \in X_2$ then $X_2 \cup \{n\} \in C_n^i $. Therefore we have proved that $Y_1 \cup Y_2 \subseteq C_n^i$.\\

Now let $Y \in C_n^i$. By Fact~\ref{fact2}, among the vertices $n, n-1,$ and $n-2$, either vertices $n$ and $n-1 \notin Y$, or vertices $n-1$ and $n-2 \notin Y$, or the vertex $n-1 \in Y$.
If the vertices $n$ and $n-1 \notin Y$, then $Y = X \cup \{n-2\}$ for some $X \in C_{n-3}^{i-1}$, when vertex $1 \in X$, then $Y \in Y_1$, and for some $X \in C_{n-1}^{i-1}$ when vertices $n-1$ and $n-2 \notin X$, then $Y \in Y_2$.
If vertices $n-1$ and $n-2 \notin Y$ , then $Y = X \cup \{n\}$ for some $X \in C_{n-1}^{i-1}$ when vertex $1 \in X$, then $Y \in Y_2$, and for some $X \in C_{n-3}^{i-1}$ when vertex $n-3 \in X$, then $Y \in Y_1$.
If vertex $n-1 \in Y$, then $Y = X \cup \{n-1\}$ for some $X \in C_{n-1}^{i-1}$ when vertices $1$ and $n-1 \notin X$, then $Y \in Y_2$, and for some $X \in C_{n-3}^{i-1}$ when vertices $1$ and $n-3 \notin X$, then $Y \in Y_1$.
Now we have proved that $C_n^i \subseteq Y_1 \cup Y_2$.
So  $C_n^i = Y_1 \cup Y_2$.

\end{enumerate}
\end{proof}

Theorem~\ref{theorem3} provides a characterization on the size of restrained dominating sets of cycles. We use Theorem~\ref{theorem3} to obtain a recursion to compute restrained dominating sets of cycles.
\begin{theorem} \label{theorem4}
For any cycle $C_n$, $|C_n^i| = |C_{n-3}^{i-1}| + |C_{n-1}^{i-1}|$.
\end{theorem}

\begin{proof}
In order to prove this theorem, we consider the following cases:
\begin{enumerate}[label=(\roman*)]
    \item From Theorem~\ref{theorem3}$(i)$, if $C_{n-1}^{i-1} = C_{n-2}^{i-1} = \phi$ and $C_{n-3}^{i-1} \neq \phi$, then $C_n^i = \big\{\{n\}\cup S_1, \{n-1\}\cup S_2, \{n-2\}\cup S_3 \mid $ \text{ vertices} 1 $ \text{and} $ 2 $ \notin S_1,\text{ vertices } $ 1 $\text{ and } $ n-3 $\notin S_2, 1\in S_3, S_1,S_2,S_3 \in C_{n-3}^{i-1}\big\}$.
    
    \item From Theorem~\ref{theorem3}$(ii)$, if $C_{n-2}^{i-1} = C_{n-3}^{i-1} = \phi$ and $C_{n-1}^{i-1} \neq \phi$, then $C_n^i=C_n^n=\big\{X \cup \{n\} \mid X \in C_{n-1}^{i-1} \big\}$.
    
    \item From Theorem~\ref{theorem3}$(iii)$, if $C_{n-1}^{i-1} \neq \phi$, $C_{n-3}^{i-1} \neq \phi$ and $C_{n-5}^{n-1} = \phi$, then \\
    $ C_n^i =  \big\{  X_1 \cup 
    \begin{cases}
    \{n-2\}, & \quad \text{ where } X_1 \in C_{n-3}^{i-1}\\
    \end{cases} \big\} \cup\\
    \Big\{  X_2 \cup 
    \begin{cases}
    \{n\}, & \quad \text{if } 1 \in X_2 \text{ or } 1 \text{ and } 2 \notin X_2 \text{, where } X_2 \in C_{n-1}^{i-1}\\
    \{n-1\}, & \quad \text{if  1 and n-1} \notin X_2 \text{, where } X_2 \in C_{n-1}^{i-1}\\
    \end{cases} \Big\}.
    $
\item From Theorem~\ref{theorem3}$(iv)$, if $C_{n-1}^{i-1} \neq \phi$ and $C_{n-3}^{i-1} \neq \phi$ then \\
    $ C_n^i =  \Big\{  X_1 \cup 
    \begin{cases}
    \{n-2\},  & \quad \text{if } 1 \in X_1 \text{, where } X_1 \in C_{n-3}^{i-1}\\
    \{n\}, & \quad \text{if  1 and 2} \notin X_1 \text{, where } X_1 \in C_{n-3}^{i-1}\\
    \{n-1\}, & \quad \text{if  1 and n-3} \notin X_1 \text{, where } X_1 in C_{n-3}^{i-1}\\
    \end{cases} \Big\} \cup\\
    \Big\{  X_2 \cup 
    \begin{cases}
    \{n\}, & \quad \text{if } 1 \in X_2 \text{ or  1 and 2 }\notin X_2 \text{, where } X_2 \in C_{n-1}^{i-1}\\
    \{n-1\}, & \quad \text{if  1 and n-1} \notin X_2 \text{, where } X_2 \in C_{n-1}^{i-1}\\
    \end{cases} \Big\}.
    $ \\
    By above construction, in every case, we have $|C_n^i| = |C_{n-1}^{i-1}| + |C_{n-3}^{i-1}|$.\\
\end{enumerate}
\end{proof}

Since computing $|C_n^i|$ using recursion (without dynamic programming) is computationally exponential, therefore we developed a technique to find $|C_n^i|$, by constructing a generating function.
\begin{theorem}\label{theorem2}
For every natural number $n \geq 4$, if $n-2\floor{\frac{n}{3}} \leq i \leq n$ then $|{C_n^i}|$ is coefficient of $x^ny^i$ in the series  expansion of the function \begin{equation}
    f(x,y)=\frac{x^4y^2(4+y^2+xy+3x^2+x^2y^2)}{1-xy-x^3y}
\end{equation}
\end{theorem}
The proof of Theorem~\ref{theorem2} is straightforward, we leave it for the readers as an exercise.

\begin{lemma} \label{lemma4} For every $n,k \in \mathbb{N}$ and $k<n$, if $\gamma_r (C_n)< k$ and $(n-k)\%2 = 1$, then $d_r (C_n, k) = 0$.
\end{lemma}

\begin{proof}
\begin{enumerate}[label=(\roman*)]By Corollary ~\ref{coroll1}(i), (ii), we know $i = \gamma_r (C_n)$ is either odd or even, if $n$ is odd or even respectively. By Proposition ~\ref{lemma2}(i), if $\gamma_r (C_n) < k$ and $(n-k)\%2 = 1$, then $d_r (C_n, k )= 0$.
\end{enumerate}
\end{proof}

\begin{lemma} \label{lemma5} Let the vertex set of $C_n$ be $V$ and $S$ be an RDS of size $k$, then every component of $V- S$ is $K_2$. Moreover, the number of $K_2$ components is $\frac{n-k}{2}$.
\end{lemma}

\begin{proof} We will prove this by method of contradiction. On the contrary, let $S$ be an RDS and $V-S$ have components that are not $K_2$. Then the components are of size $1 \text{ or greater than } 2$. In the former case, there are no neighbours in $V-S$, while in the latter case, there exists at least one vertex such that it has no neighbours in $S$, which are contradictions to the definition of RDS. Since every component of $V - S$ forms a $K_2$ and $|V - S| = n - k$, therefore, there are exactly $\frac{n-k}{2}$ numbers of $K_2$ components . 
\end{proof}

\section{Restrained Domination Polynomial of a Cycle}\label{section4}
In this section, we show how to compute the restrained domination polynomial of cycles using the recursion of Theorem~\ref{theorem4}. For completion purpose, we restate the definition of restrained  domination polynomial of cycle.

The restrained  domination polynomial of cycle $C_n$, $D_r(C_n,x)$  is defined as

\[ 
    D_r(C_n, x) =\sum_{i=n-2\floor{\frac{n}{3}}}^{n} d_r(C_n, i)x^i
    \]

We now show the recursive computation of RDP of a cycle.
\begin{theorem} \label{theorem5}
For every $n \geq 4$,
   \[ D_r(C_n,x)=x[D_r(C_{n-3},x) + D_r(C_{n-1},x)],\]
with initial values $D_r(C_1,x)=x$, $D_r(C_2,x)=x^2$, $D_r(C_3,x)=3x+x^3$.
\end{theorem}

Using Theorem~\ref{theorem5}, we obtain the coefficients of $D(C_n, x)$ for $1 \leq n \leq 23$ in Table~\ref{ctable} and verify it by Theorem~\ref{theorem2}. Let
$d(C_n, j) = |C_n^j|$. We found a few relationships between the numbers $d(C_n, j) \text{ and } \gamma_r(C_n) \leq j \leq n$ which is shown in Table ~\ref{ctable}. The rows in Table~\ref{ctable} represent the orders of cycles and the columns represents the cardinality $i$ of restrained dominating sets. The coefficient of $x^i$ in the RDP of a cycle $C_n$ is the $(n,i)$-th entry in Table~\ref{ctable}.

\begin{table}[h]

n
$\updownarrow$
\begin{tabular}{c|ccccccccccccccccccccccccccc}
    & & & & & & & & $\leftarrow $& i& $\rightarrow $& & & & & & & & & &\\
  & 1 & 2 & 3 & 4 & 5 & 6 & 7 & 8 & 9 & 10 & 11 & 12 & 13 &14 & 15 & 16 & 17 & 18 & 19 & 20 & 21 & 22 & 23\\
 \hline

1 & 1 &  &  &  \\
2 & 0 & 1 &  &  \\
3 & 3 & 0 & 1 &  & & & & & & & & & & & \\
4 & 0 & 4 & 0 & 1 & & & & & & & & & & \\
5 & 0 & 0 & 5 & 0 & 1 & & & & & & & & & & & &\\
6 & 0 & 3 & 0 & 6 & 0 & 1 & & & & & & & & & & &\\
7 & 0 & 0 & 7 & 0 & 7 & 0 & 1 & & & & & & & & & &\\
8 & 0 & 0 & 0 & 12 & 0 & 8 & 0 & 1 & & & & & & & & &\\
9 & 0 & 0 & 3 & 0 & 18 & 0 & 9 & 0 & 1 & & & & & & & &\\
10 & 0& 0& 0& 10& 0& 25& 0& 10& 0& 1& & & & & & &\\
11 & 0& 0& 0& 0& 22& 0& 33& 0& 11& 0& 1&  & & & &\\
12 & 0& 0& 0& 3& 0& 40& 0& 42& 0& 12& 0& 1& & & & &\\
13 & 0& 0& 0& 0& 13& 0& 65& 0& 52& 0& 13& 0& 1& & & & &\\
14 & 0& 0& 0& 0& 0& 35& 0& 98& 0& 63& 0& 14& 0& 1& & \\
15 & 0& 0& 0& 0& 3& 0& 75& 0& 140& 0& 75& 0& 15& 0& 1& & \\
16 & 0& 0& 0& 0& 0& 16& 0& 140& 0&  192& 0& 88& 0& 16& 0& 1& &\\
17 & 0& 0& 0& 0& 0& 0& 51& 0& 238& 0& 255& 0& 102& 0& 17& 0& 1& \\
18 & 0& 0& 0& 0& 0& 3& 0& 126& 0& 378& 0& 330& 0& 117& 0& 18& 0& 1&\\  
19 & 0& 0& 0& 0& 0& 0& 19& 0& 266& 0& 570& 0& 418& 0& 133& 0& 19& 0& 1& &\\
20 & 0& 0& 0& 0& 0& 0& 0& 70& 0& 504& 0& 825& 0& 520& 0& 150& 0& 20& 0& 1& &\\
21 & 0& 0& 0& 0& 0& 0& 3& 0& 196& 0& 882& 0& 1155& 0& 637& 0& 168& 0& 21& 0& 1& \\ 
22 & 0& 0& 0& 0& 0& 0& 0& 22& 0& 462& 0& 1452& 0& 1573& 0& 770& 0& 187& 0& 22& 0& 1& &\\ 
23 & 0& 0& 0& 0& 0& 0& 0& 0& 92& 0& 966& 0& 2277& 0& 2093& 0& 920& 0& 207& 0& 23& 0& 1& &\\ \\
\end{tabular}
\caption{ $d_r(C_n,i)$, the number of restrained dominating set of of cycle $C_n$ with cardinality $i.$}\label{ctable}
\end{table}

Next we characterize the coefficients of RDP of cycles and find a few relations between the coefficients of RDPs of cycles of different orders.
\begin{obsn} The following properties holds for coefficients of $D_r(C_n,x)$, for every $n \in \mathbb{N}$:
\begin{enumerate}[label=(\roman*)]
    \item $d_r(C_{3n},n) = 3$.
    
    \item If $n \geq 3$, then $d_r(C_n,n)=1$.
    
    \item If $n \geq 3$, then $d_r(C_n,n-1)=0$.
    
    \item If $n \geq 3$, then $d_r(C_n,n-2)=n$.
\end{enumerate}
\end{obsn}
\begin{theorem} \label{theorem6}
The following properties holds for the coefficients of $D_r(C_n,x)$, $\forall n \in \mathbb{N}$:
\begin{enumerate}[label=(\roman*)]
    \item If $n \geq 4$ and $i \geq n-2\floor{\frac{n}{3}}$, then $d_r(C_{n},i) = d_r(C_{n-3},i-1) + d_r(C_{n-1},i-1)$.
    \item $d_r(C_{3n+1},n+1)=3n+1$.
    
    \item If $n \geq 2$, then $d_r(C_{3n-1},n+1) = \frac{n}{2}(3n-1)$.
        
    \item If $n \geq 5$, then $d_r(C_n,n-4)=\frac{n}{2}(n-5)$.
    
    \item If $n \geq 4$, then $\sum_{i=n}^{3n} d_r(C_i,n) = 2\sum_{i=n-1}^{3(n-1)} d_r(C_i,n-1)$.
    
    \item For $n \geq 3 $, 
    \begin{itemize}
        \item if n is even, then
        \[ \begin{cases}
        d_r(C_i,n) < d_r(C_{i+2},n), & \quad\text{ for }n \leq i \leq 2n - 2 \\
        d_r(C_i,n) > d_r(C_{i+2},n), & \quad\text{ for }2n \leq i \leq 3n-2 \\
           \end{cases}
        \]
        \item if n is odd, then \[ \begin{cases}
        d_r(C_i,n) < d_r(C_{i+2},n), & \quad\text{ for }n \leq i \leq 2n - 1 \\
        d_r(C_i,n) > d_r(C_{i+2},n), & \quad\text{ for }2n+1 \leq i \leq 3n-2 \\
        
        \end{cases} \]
    \end{itemize}
    
     \item If $S_n=\sum_{i=\gamma_r(C_n)}^{n} d_r(C_n,i)$, then for every $n \geq 4$, $S_n = S_{n-1} + S_{n-3}$ with initial values $S_1 = 1, S_2 = 1, S_3 = 4$.

\end{enumerate}
\end{theorem}
   
\begin{proof}
\begin{enumerate}[label=(\roman*)]
    
    \item It follows from Theorem~\ref{theorem4}.
    
    \item We prove it by induction on $n$. The result is true for $n=1$, because $C_4^2 = \big\{ \{1, 2\}, \{2, 3\}, \{3, 4\}, \{4, 1\} \big\}$. Therefore, $d_r(C_4,2)=|C_4^2|=4$.
    Now let us suppose the result is true for all natural numbers less than $n$, and we have to prove it for $n$. By $(i), (ii)$ and induction hypothesis, we have
    \begin{align*}
        & d_r(C_{3n+1},n+1) \\
        &= d_r(C_{3n},n) + d_r(C_{3n-2},n) \\
        &= 3 + 3(n-1) + 1 \\
        &= 3n + 1.
        \end{align*}
    Therefore, by method of induction $d_r(C_{3n+1},n+ 1) = 3n+1$,$ \forall n \in \mathbb{N}$.

    \item We prove it by induction on $n$. Since $C_5^3=\big\{ \{1,2,3\}, \{2,3,4\}, \{3,4,5\}, \{4,5,1\}, \{5,1,2\} \big\}$, so $d_r(C_5,3) = 5$. The result is true for $n = 2$. Now let us suppose the result is true for all natural numbers less than n, and we have to prove it for $n$. By $(ii), (iii)$ and induction hypothesis, we have
    \begin{align*}
    & d_r(C_{3n-1},n+1) \\
    &= d_r(C_{3n-2},n) + d_r(C_{3n-4},n) \\
    &= 3(n-1)+1 + \frac{n-1}{2}(3(n-1)-1) \\
    &= \frac{n}{2}(3n-1).
    \end{align*}
    Therefore, by method of induction $d_r(C_{3n-1},n+ 1)=\frac{n}{2}(3n-1)$, $\forall n \geq 2$.
    
    
    \item We prove it by induction on $n$. The result is true for $n=5$, because $C_5^1 = \phi$, so $d_r(C_5,1) = 0$. Now let us suppose the result is true for all natural numbers less than $n$, and we have to prove it for $n$. By $(ii), (vii)$ and induction hypothesis, we have
    
    \begin{align*}
    & d_r(C_{n},n-4) \\
    &= d_r(C_{n-1},n-5) + d_r(C_{n-3},n-5) \\
    &= \frac{(n-1)}{2}(n-1-5) + n-3 \\
    &= \frac{n}{2}(n-5).
    \end{align*}
    
    Therefore, by method of induction $d_r(C_{n},n-4)=\frac{n}{2}(n-5)$, $\forall n \geq 5$.
    \item We prove it by induction on $n$. Firstly, let  $n = 4$, then $\sum_{i=4}^{12}d_r(C_i,4) = 32 = 2 \sum_{i=3}^{9} d_r(C_i,3)$. Now suppose the result is true for all natural numbers less than $n$, and we have to prove it for $n$. By $(ii), (vii)$ and induction hypothesis, we have
    \begin{align*}
    &\sum_{i=n}^{3n}d_r(C_i,n)\\
    &= \sum_{i=n}^{3n}d_r(C_{i-1},n-1) + \sum_{i=n}^{3n}d_r(C_{i-3},n-1) \\
    &= 2\sum_{i=n-1}^{3(n-1)}d_r(C_{i-1},n-2) + 2\sum_{i=n-1}^{3(n-1)}d_r(C_{i-3},n-2) \\
    &= 2\sum_{i=n-1}^{3(n-1)}d_r(C_{i},n-1) 
    \end{align*}
    Therefore, by method of induction $\sum_{i=n}^{3n}d_r(C_i,n)= 2\sum_{i=n-1}^{3(n-1)}d_r(C_{i},n-1)$, $\forall n \geq 4$.
    
    \item We prove it by induction on $n$. The result holds for $n=3$. Now suppose the result is true for all natural number less than $n$, and we have to prove it for $n$.
    To prove our claim, consider the following cases when $n$ is even:\\
    \begin{itemize}
        \item Case 1:For $n \leq i \leq 2n-2:$ \\
        By Theorem~\ref{theorem4} and induction hypothesis, we have 
                \begin{align*}
        & d_r(C_i,n) \\
        &= d_r(C_{i-1},n-1) + d_r(C_{i-3},n-1) \\
        & < d_r(C_{i+1},n-1) + d_r(C_{i-1},n-1) \\
        &= d_r(C_{i+2},n)
    \end{align*}

        \item Case 2: For $2n \leq i \leq 3n-2:$\\
        By Theorem~\ref{theorem4} and induction hypothesis, we have
        \begin{align*}
        & d_r(C_i,n) \\
        &= d_r(C_{i-1},n-1) + d_r(C_{i-3},n-1) \\
        & > d_r(C_{i+1},n-1) + d_r(C_{i-1},n-1) \\
        &= d_r(C_{i+2},n)
    \end{align*}
    Therefore, the statement holds when $n$ is even.
    \end{itemize}
    Similarly, we can prove for odd $n$.
    Therefore, by method of induction the proposition holds for all $n \geq 3$.
        \item By induction on $n$. Since we have $S_1=1, S_2=1, S_3=4$, so for $n=4$, $S_4 = 4 = S_3 +S_1$. Now let us suppose the result is true for all natural number less than $n$, and we have to prove it for $n$. By Theorem~\ref{theorem4} and induction hypothesis, we have
    \begin{align*}
        & S_n = \sum_{i=n-2\floor*{\frac{n}{3}}}^{n} d_r(C_n,i) \\
        &= \sum_{i=n-2\floor*{\frac{n}{3}}}^{n} (d_r(C_{n-1},i-1) + d_r(C_{n-3},i-1)) \\
        &= \sum_{i=n-1-2\floor*{\frac{n}{3}}}^{n-1} d_r(C_{n-1},i) +  \sum_{i=n-1-2\floor*{\frac{n}{3}}}^{n-3} d_r(C_{n-3},i) \\
        &= S_{n-1} + S_{n-3}.
    \end{align*}
    Therefore, by method of induction $S_n = S_{n-1} + S_{n-3}$, $\forall n \geq 4$.
\end{enumerate}
\end{proof}

\begin{theorem} \label{theorem7}
Total number of terms in $D_r(C_n,x)=1+\floor*{\frac{n}{3}}$.
\end{theorem}
\begin{proof}
Total number of terms in $D_r(C_n,x)$
\begin{align*}
&= (n-\gamma_r(C_n))/2 + 1 \\
&=\frac{n-(n-2\floor{\frac{n}{3}})}{2} + 1 \\
&=\floor{\frac{n}{3}} + 1
\end{align*}

\end{proof}

\section{Conclusion}\label{section5}
In this paper, we introduced two novel methods for computing RDP and enumerating RDS of a cycle, namely, a recurrence on computing the restrained domination number and a generating function. We also presented a few characterizations on the restrained dominating sets of cycles. We listed several properties of the coefficients of RDP of cycles and reported empirical data on cycles with order up to $23$ in Table~\ref{ctable}. An immediate extension of this work is to find polynomial time algorithms to compute RDS, for subclasses of special classes of graphs, such as bipartite graphs and chordal graphs. Our approach maybe applied to graphs with fixed size cycles in it. In addition, another direction for this study is to find the complexity class of computing RDP of a graph.

\bibliographystyle{plain}
\bibliography{myfile} 
\end{document}